\documentclass[11pt]{amsart}
\usepackage[utf8]{inputenc}
\usepackage[margin=1.25in]{geometry}
\usepackage{amsmath,amssymb,amsthm}
\usepackage{tikz}
\usepackage{tikz-cd}    % commutative diagrams with tikz
\usepackage{commath}    % formatting formulas
\usepackage{mathtools}    % formatting formulas
\usepackage{graphicx}
\usepackage{etoolbox}
\usepackage{enumitem}
\usepackage{comment}
\usepackage{stmaryrd}
\usepackage{hyperref}
\usepackage{subfiles}
\usepackage{mathrsfs}

\usepackage[
backend=biber,
style=alphabetic,
maxalphanames=10,
maxbibnames=10,
sorting=anyt,
]{biblatex}

\theoremstyle{plain}
\newtheorem{thm}[subsubsection]{Theorem}
\newtheorem{lem}[subsubsection]{Lemma}
\newtheorem{prop}[subsubsection]{Proposition}
\newtheorem{cor}[subsubsection]{Corollary}

\theoremstyle{definition}
\newtheorem{defn}[subsubsection]{Definition}

\theoremstyle{remark}
\newtheorem{rem}[subsubsection]{Remark}

\newcommand{\ZZ}{\mathbb{Z}}
\newcommand{\NN}{\mathbb{N}}
\newcommand{\QQ}{\mathbb{Q}}
\newcommand{\RR}{\mathbb{R}}
\newcommand{\CC}{\mathbb{C}}

\newcommand{\FF}{\mathbb{F}}

\newcommand{\Acal}{\mathcal{A}}

\newcommand{\Ecal}{\mathcal{E}}
\newcommand{\Fcal}{\mathcal{F}}

\newcommand{\Rcal}{\mathcal{R}}

\newcommand{\Ucal}{\mathcal{U}}
\newcommand{\Vcal}{\mathcal{V}}

\newcommand{\Frob}{\operatorname{Frob}}

\newcommand{\ord}{\operatorname{ord}}

\newcommand{\trdeg}{\operatorname{trdeg}}

\newcommand{\ev}{\operatorname{ev}}

\newcommand{\ovl}{\overline}
\newcommand{\udl}{\underline}
\newcommand{\td}{\tilde}
\newcommand{\isoto}{\mathrel{\mathop{\longrightarrow}\limits^{\vbox to 0ex{\vss\hbox{$\scriptstyle\sim$}\vskip -2pt}}}}

\newcommand{\sbe}{\subseteq}

\let\oldforall\forall
\renewcommand{\forall}{\oldforall \: }
\let\oldexist\exists
\renewcommand{\exists}{\oldexist \: }

\newcommand{\wideand}{\quad \text{and} \quad}
\newcommand{\widewhere}{\quad \text{where} \quad}
\newcommand{\wideforall}{\quad \text{for all} \quad}
\newcommand{\widefor}{\quad \text{for} \quad}

\newcommand{\Fq}{\FF_q}
\newcommand{\Fqd}{\FF_{q^d}}

\newcommand{\T}{\theta}

\newcommand{\bfy}{\mathbf{y}}

\newcommand{\iaf}{\Pi_\infty^{\textnormal{ari}}}
\newcommand{\biaf}{\boldsymbol{\Pi}_{\infty}^{\textnormal{ari}}}
\newcommand{\fia}{f_\infty}

\newcommand{\vaf}{\Pi_v^{\textnormal{ari}}}
\newcommand{\bvaf}{\boldsymbol{\Pi}_{v}^{\textnormal{ari}}}

\newcommand{\ags}{G^{\textnormal{ari}}}
\newcommand{\fG}{f_{G,\l}}

\newcommand{\bfG}{\mathbf{G}^{\textnormal{ari}}}

\newcommand{\pinf}{\partial_\infty}
\newcommand{\pv}{\partial_v}

\newcommand{\Aari}{\Acal^{\textnormal{ari}}_\l}
\newcommand{\Rari}{\Rcal^{\textnormal{ari}}_\l}
\newcommand{\Fari}{\Fcal^{\textnormal{ari}}_\l}
\newcommand{\Uari}{\Ucal^{\textnormal{ari}}_\l}
\newcommand{\Vari}{\Vcal^{\textnormal{ari}}_\l}

\newcommand{\EariQ}{\Ecal^{\textnormal{ari}}_{\l,\QQ}}
\newcommand{\FariQ}{\Fcal^{\textnormal{ari}}_{\l,\QQ}}
\newcommand{\UariQ}{\Ucal^{\textnormal{ari}}_{\l,\QQ}}

\let\l\ell

\DeclarePairedDelimiter{\ang}{\langle}{\rangle}

\newcommand{\lrfloor}[1]{\left\lfloor #1 \right\rfloor}

\newcommand{\lra}{\longrightarrow}
\newcommand{\ltra}{\relbar\joinrel\twoheadrightarrow}
\tikzcdset{arrow style=math font}

\hypersetup{
colorlinks=true,
linkcolor=blue,
urlcolor=black,
citecolor=red,
linktoc=page
}

\title[On Special Arithmetic Gamma Values and Thakur's Gauss Sums]{On Special Arithmetic Gamma Values and Thakur's Gauss Sums: From $\infty$-adic to $v$-adic}
\author{Ting-Wei Chang}
\date{\today}

\thanks{This work is supported by the National Science and Technology Council grant nos. 114-2628-M-007-003, 115-2628-M-007-008, and 115-2917-I-007-003.}

\begin{document}

	\begin{abstract}
		In this paper, we establish an explicit algebra isomorphism between the normalized $\infty$-adic arithmetic gamma values and their $v$-adic counterparts.
		The construction relies on Thakur’s arithmetic Gauss sums, which arise naturally from the Gross-Koblitz-Thakur relations.
		Our result demonstrates the phenomenon that Gauss sums, commonly regarded as “finite” analogs of gamma functions, serve as a bridge connecting the $\infty$-adic and $v$-adic gamma values.
	\end{abstract}
	
	\maketitle
	
	\section{Introduction}
	
	\subsection{Motivation}
	
	The values of the classical Euler gamma function at non-integral rational arguments possess deep connections with CM periods, as demonstrated by the Chowla-Selberg formula \cite{sc1967epstein} (see also \cite{gross1978ontheperiods}).
	Conjecturally, all these values are transcendental over $\QQ$, with specific cases established by Chudnovsky \cite{chudnovsky1984contributions}.
	In this direction, the folklore Lang-Rohrlich conjecture asserts that all $\ovl{\QQ}$-algebraic relations among these special values are explained by the standard functional equations satisfied by the Euler gamma function.
	
	On the other hand, Ogus' $p$-adic analog of the Chowla-Selberg formula \cite{ogus1990p-adic} relates the values of Morita's $p$-adic gamma function \cite{morita1975padic} at $p$-integral rational numbers to $p$-adic CM periods.
	Furthermore, the Gross-Koblitz formula \cite{gk1979gauss} expresses Frobenius-twisted products of such values in terms of Gauss sums, which yields additional algebraic relations among them.
	Together with the standard functional equations, these relations are conjectured to account for all algebraic relations among these special $p$-adic gamma values.
	
	The analogy between Euler and Morita $p$-adic gamma values naturally raises the problem of establishing explicit connections between them.
	An answer to this question was provided by Deligne's theorem on special Euler gamma values together with Gross-Koblitz's reformulation.
	Specifically, Deligne (see \cite[Theorem 7.15]{deligne1982hodge}) considered particular Euler gamma monomials (i.e., an element in the subgroup of $\CC^\times$ generated by special Euler gamma values and $2\pi\sqrt{-1}$).
	Using Hodge theory on Fermat hypersurfaces, he proved the algebraicity of such monomials and interpreted the actions of the geometric Frobenius on these numbers in terms of Gauss sums and Hecke characters (see also \cite[Theorem 4.5]{gk1979gauss}).
	This result enabled Gross-Koblitz to relate these Euler gamma monomials to their $p$-adic counterparts, which yields a direct connection between Archimedean and $p$-adic special gamma values.
	We refer the reader to \cite[Theorem 4.6]{gk1979gauss} for the precise formula.
	
	This particular connection motivates us to study a similar phenomenon over function fields in positive characteristic.
	In the present paper, we focus on the \textit{arithmetic gamma functions} introduced by Goss \cite[Appendix]{goss1980modular}, and take the investigation further by connecting the algebras generated by all special values of these functions.
	
	\subsection{The main result}
		
	Throughout this paper, let $p>0$ be any prime number and $q$ be a power of $p$.
	We consider the polynomial ring $A := \Fq[\T]$ over a finite field $\Fq$ of $q$ elements with field of fractions $k := \Fq(\T)$.
	Let $k_\infty$ be the completion of $k$ at the infinite place with a fixed algebraic closure $\ovl{k}_\infty$, and $\ovl{k}$ be the algebraic closure of $k$ inside $\ovl{k}_\infty$.
	On the other hand, for a monic irreducible polynomial $v \in A$, we let $k_v$ be the completion of $k$ at $v$ with a fixed algebraic closure $\ovl{k}_v$, and choose an identification from $\ovl{k}$ into $\ovl{k}_v$.
	
	By interpolating the Carlitz factorials at all places of $k$, Goss \cite[Appendix]{goss1980modular} defined the \textit{$\infty$-adic} (resp. \textit{$v$-adic}) \textit{arithmetic gamma function} $\iaf: \ZZ_p \to k_\infty^\times$ (resp. $\vaf: \ZZ_p \to k_v^\times$).
	Since their introduction, these functions have been extensively studied in several directions, particularly regarding the functional equations and the transcendence of their special values at $p$-integral rational numbers.
	We refer the reader to \S\ref{subsec-arithmetic-gamma-functions} for precise definitions and a historical overview.
	Among these results is Thakur's analog of Deligne's theorem \cite[Theorem 4.6]{thakur1991gamma}, which leads to a direct connection between special $\infty$-adic arithmetic gamma monomials and their $v$-adic counterparts.
	We recall his result in \S\ref{subsubsec-a-comparison-with-thakur's-formula}.
	
	The main goal of this paper is to generalize this phenomenon from individual monomials to the algebras generated by all special arithmetic gamma values.
	This motivates us to introduce the \textit{normalized arithmetic gamma values}
	$$
	\biaf(y) \in \ovl{k}_\infty^\times
	\wideand
	\bvaf(y) \in \ovl{k}_v^\times
	\widefor
	y \in \QQ \cap \ZZ_p
	$$
	defined by modifying $\iaf(y)$ and $\vaf(y)$ by suitable algebraic factors (see Definitions \ref{defn-biaf} and \ref{defn-bvaf}).
	Then our main theorem can be stated as follows.
	
	\begin{thm}[see Theorem \ref{thm-main}]    \label{thm-main-thm-intro}
		For each $\l \in \NN$, there exists an explicit $\ovl{k}$-algebra isomorphism
		$$
		\bfG: \ovl{k} \left[\biaf(y) \Bigm| y \in \frac{1}{q^\l-1} \ZZ \right]
		\isoto
		\ovl{k} \left[\bvaf(y) \Bigm| y \in \frac{1}{q^\l-1} \ZZ \right]
		$$
		given by $\bfG(\biaf(y)) := \bvaf(y) \ags(y)$, where the map $\ags: \QQ \cap \ZZ_p \to \ovl{k}_v^\times$ is defined through Thakur's arithmetic Gauss sums (see \eqref{eq-gauss-sum}).
	\end{thm}
	
	Note that this result demonstrates the phenomenon that Gauss sums, commonly regarded as “finite” analogs of gamma functions, serve as a bridge connecting the $\infty$-adic and $v$-adic gamma values.
	We also compare this result with Thakur's relation for special arithmetic gamma monomials in \S\ref{subsubsec-a-comparison-with-thakur's-formula}.
	
	\begin{rem}      \label{rem-geo-and-two}
		Besides the arithmetic case studied in this paper, there are two other analogs of gamma functions in positive characteristic, the \textit{geometric} and the \textit{two-variable} gamma functions (see \cite[Section 5]{thakur1991gamma} and \cite[Subsection 9.9]{goss1996basic}).
		Together with the notion of geometric Gauss sums introduced by the author \cite{chang2025geometric}, we believe that the phenomenon demonstrated in Theorem \ref{thm-main-thm-intro} also occurs in these cases.
		This problem will be investigated in future work.
	\end{rem}
	
	\begin{rem}
		The study of connecting special values in the $\infty$-adic and $v$-adic sides goes far beyond the context of gamma functions.
		For instance, it was conjectured by Furusho \cite{furusho2006multiple}, \cite{furusho2007p-adic} that there exists a well-defined surjective $\QQ$-linear map which sends the classical real multiple zeta values to their $p$-adic counterparts.
		In the function field setting, a stronger version of an analogous conjecture was proved by Chang-Mishiba \cite{cm2021conjecture}, who established a well-defined surjective $\ovl{k}$-linear homomorphism sending the $\infty$-adic multiple zeta values to their $v$-adic counterparts.
		It would be interesting to investigate such a question for other families of special values.
	\end{rem}
	
	\subsection{Outline of the paper}
	
	This paper is organized as follows.
	In \S\ref{subsec-arithmetic-gamma-functions}, we review the definitions and the transcendence results concerning the arithmetic gamma functions.
	We then introduce the normalizations of their special values in \S\ref{subsec-normalized-arithmetic-gamma-values} and state the main result in \S\ref{subsec-the-main-result}.
	
	The remaining part of the paper (\S\ref{sec-proof-of-the-main-result}) is devoted to proving the main theorem.
	Following preliminary lemmas (\S\ref{subsec-preliminary-lemmas}), we establish the standard monomial relations among normalized arithmetic gamma values in \S\ref{subsec-monomial-relations-of-normalized-gamma-values}.
	Using the notion of “gamma distributions”, we show in \S\ref{subsec-transcendence-of-normalized-gamma-values} that these relations account for all algebraic relations among normalized gamma values.
	Finally, in \S\ref{subsec-the-algebra-isomorphism}, we utilize these transcendence results to identify the $\ovl{k}$-algebras in question and conclude the proof of the desired isomorphism.
	
	\section{Normalized arithmetic gamma values}
	
	\subsection{Arithmetic gamma functions}   \label{subsec-arithmetic-gamma-functions}
	
	\subsubsection{Carlitz factorials}      \label{subsubsec-carlitz-factorials}
	
	In \cite{carlitz1935oncertain}, Carlitz introduced the notion of factorial function in positive characteristic as follows.
	For $i \geq 0$, let $D_i$ be the product of all monic polynomials in $A$ of degree $i$.
	Then for $n = \sum_{i=0}^{s} n_i q^i \in \NN$ with $0 \leq n_i < q$ for all $i$, the \textit{Carlitz factorial} is defined as
	$$
	\Pi(n) := \prod_{i=0}^s D_i^{n_i} \in A.
	$$
	This is regarded as an analog of the usual factorial function in view of Sinnott's result (see \cite[Theorem 9.1.1]{goss1996basic}) that
	$$
	\ord_f(\Pi(n)) = \sum_{e \geq 1} \lrfloor{\frac{n}{q^{e \deg f}}}
	$$
	for any monic irreducible polynomial $f \in A$.
	
	\subsubsection{Arithmetic gamma functions}     \label{subsubsec-arithmetic-gamma-functions}
	
	In \cite[Appendix]{goss1980modular}, Goss gave Morita-style interpolations of the Carlitz factorials at all places of $k$, which are now called the arithmetic gamma functions.
	In the $\infty$-adic case, we put $\ovl{a} := a \cdot \T^{-\deg a} \in k$ for any monic polynomial $a \in A$, so that $\ovl{a}$ is a one-unit in $k_\infty$.
	Then the \textit{$\infty$-adic arithmetic gamma function} $\iaf: \ZZ_p \to k_\infty^\times$ is defined as
	\begin{equation}      \label{eq-iaf}
		\iaf\left( \sum_{i=0}^\infty y_iq^i \right) := \prod_{i=0}^{\infty} \ovl{D}_i^{y_i}
	\end{equation}
	where $0 \leq y_i < q$ for all $i$.
	
	In the $v$-adic setting, for $i \geq 0$, we let $\td{D}_i$ be the product of all monic polynomials in $A$ of degree $i$ which are relatively prime to $v$.
	Then the \textit{$v$-adic arithmetic gamma function} $\vaf: \ZZ_p \to k_v^\times$ is defined as 
	\begin{equation}         \label{eq-vaf}
		\vaf\left( \sum_{i=0}^\infty y_iq^i \right) := \prod_{i=0}^{\infty} (-\td{D}_i)^{y_i}
	\end{equation}
	where $0 \leq y_i < q$ for all $i$.
	
	\begin{rem}
		Following \cite[\nopp 0.4]{thakur1991gamma}, one also defines the arithmetic gamma functions by shifting \eqref{eq-iaf} and \eqref{eq-vaf} by one, referring to the objects defined above as the \textit{arithmetic factorial functions}.
		In this paper, we adopt the original terminology of Goss \cite{goss1980modular} and abbreviate them to gamma functions (recall Remark \ref{rem-geo-and-two}).
	\end{rem}
	
	\subsubsection{Functional equations}
	
	The standard functional equations of these two gamma functions, namely the translation, reflection, and multiplication formulas, were well-established by Thakur in \cite{thakur1991gamma}.
	In fact, since both of them are functions on $\ZZ_p$ of the form $f(\sum y_i q^i) := \prod f_i^{y_i}$, these identities were proved formally through $p$-adic digit manipulations.
	We refer the reader to Theorem \ref{thm-general-fe} (and Remark \ref{rem-functional-eq}) for the precise formulas.
	
	\subsubsection{Transcendence of \texorpdfstring{$\infty$}{infty}-adic gamma values}
	
	We now consider the $\infty$-adic and $v$-adic gamma functions evaluated at $y \in \QQ \cap \ZZ_p$, called the special gamma values, and their corresponding transcendence results.
	
	In the $\infty$-adic case, it follows immediately from the definition that $\iaf(n) \in k$ for all $n \in \ZZ_{\geq 0}$.	
	On the other hand, by relating the special gamma values with Carlitz period, Thakur \cite{thakur1991gamma} showed that $\iaf(n)$ is transcendental over $k$ for $n \in \ZZ_{<0}$ and obtained an analog of the Chowla-Selberg formula.
	Furthermore, using the automata methods, he \cite{thakur1996transcendence} also established the transcendence of $\iaf(q^i/(1-q^\l))$ for $\l \in \NN$ and $0 \leq i < \l$.
	This result was later generalized by Allouche \cite{allouche1996transcendence}, who showed that $\iaf(y)$ is transcendental for all $y \in (\QQ \cap \ZZ_p) \setminus \ZZ_{\geq0}$.
	
	Turning to their algebraic relations, Chang-Papanikolas-Thakur-Yu \cite{cpty2010algebraic} showed that for each $\l \in \NN$,
	$$
	\trdeg_{\ovl{k}} \ovl{k}\left( \iaf(y) \Bigm| y \in \frac{1}{q^\l-1}\ZZ \right)
	= \l.
	$$
	Consequently, this implies that all algebraic relations among special $\infty$-adic gamma values are generated by the standard functional equations listed in Theorem \ref{thm-general-fe}.
	Thus, an analog of the Lang-Rohrlich conjecture is proved in this setting.
	We mention that this result was later reproved by Wei \cite{wei2026algebraic} using the notion of distribution relations and the theory of Shimura's conjecture on period symbols developed in \cite{bcpw2022function}.
	
	A key step of this classification involves period interpretations of the special gamma values.
	More precisely, they considered the Carlitz $\FF_{q^\l}[t]$-module (where $t$ is another variable) and regarded it as a Drinfeld $\Fq[t]$-module of rank $\l$ with CM by $\FF_{q^\l}[t]$.
	Then together with Thakur's analog of the Chowla-Selberg formula, they showed that the periods of its associated dual $t$-motive are product of special $\infty$-adic gamma values up to algebraic multiples.
	In \cite{wei2026algebraic}, these values also appear in the pairing comparison between the Betti and de Rham modules as period symbols.
	
	\subsubsection{Transcendence of \texorpdfstring{$v$}{v}-adic gamma values}
	
	In the $v$-adic case, one naturally encounters the Gross-Koblitz phenomenon, which provides extra algebraic relations among special $v$-adic gamma values.
	This is due to the work of Thakur \cite{thakur1988gauss}, where he introduced an analog of Gauss sums over function fields as follows.
	
	Put $d := \deg v$.
	For $j \in \ZZ_{\geq0}$, we let $\chi_j: A/v \to \Fqd \sbe \ovl{k}_v$ be an $\Fq$-algebra homomorphism indexed so that $\chi_j^q = \chi_{j+1}$ and $\chi_0$ is the Teichmüller character.
	Let $\Lambda_v$ be the $v$-torsions of the Carlitz module $C(\ovl{k}_v)$, and choose an $A$-module isomorphism $\psi: A/v \to \Lambda_v$ (see \cite[Section 3]{goss1996basic}).
	Then Thakur's \textit{arithmetic Gauss sums} are defined as
	\begin{equation}     \label{eq-gauss-sum}
		g_j
		:= -\sum_{z \in (A/v)^\times} \chi_j(z^{-1}) \psi(z) \in \ovl{k}_v.
	\end{equation}
	We note that $g_j$ only depends on $j$ modulo $d$.
	
	By Hensel's lemma, there exists a unique $(q^d-1)$-st root $\lambda \in k_v(\psi(1))$ of $-v$ such that $\lambda \equiv \psi(1) \pmod{\psi(1)^2}$.
	Then under this setting, Thakur proved the following analog of the Gross-Koblitz formula: for each $0 \leq j < d$, we have
	\begin{equation}      \label{eq-gk-formula}
		\vaf\left(\frac{q^j}{1-q^d}\right) = -\frac{\lambda^{q^j}}{g_j}.
	\end{equation}
	
	It is natural to ask whether a Lang-Rohrlich phenomenon holds in this setting, that is, whether all algebraic relations among special $v$-adic gamma values are explained by the standard functional equations listed in Theorem \ref{thm-general-fe}, together with the extra relations coming from Thakur's Gross-Koblitz formula \eqref{eq-gk-formula}.
	This conjecture was confirmed by the work of Chang-Wei-Yu \cite{cwy2026vadic}, where they showed that for each $\l \in \NN$,
	$$
	\trdeg_{\ovl{k}} \ovl{k}\left( \vaf(y) \Bigm| y \in \frac{1}{q^\l-1}\ZZ \right)
	= \l - \gcd(d,\l).
	$$
	
	As in the $\infty$-adic case, a key step of this result involves period interpretations of the special $v$-adic gamma values.
	Specifically, by considering the \textit{same} $t$-motive as in the $\infty$-adic case and applying the theory developed by Hartl-Kim \cite{hk2020local}, they obtained a comparison isomorphism between the crystalline and de Rham modules.
	Next, they showed that the special values in question appear in the matrix representation of the Frobenius action on the de Rham module (via the natural “isocrystal” action on the crystalline module).
	This is regarded as an analog of Ogus' $p$-adic Chowla-Selberg formula.
	
	\subsection{Normalized arithmetic gamma values}       \label{subsec-normalized-arithmetic-gamma-values}
	
	\subsubsection{Prime-to-\texorpdfstring{$v$}{v} Carlitz factorials}      \label{subsubsec-prime-to-v-carlitz-factorials}
	
	In order to relate the special $\infty$-adic and $v$-adic gamma values, we consider the following \textit{prime-to-$v$ Carlitz factorials}.
	For $n = \sum_{i=0}^{s} n_i q^i \in \NN$ with $0 \leq n_i < q$ for all $i$, define
	$$
	\td{\Pi}(n) := \prod_{i=0}^s \td{D}_i^{n_i} \in A,
	$$
	where one recalls \S\ref{subsubsec-arithmetic-gamma-functions} that $\td{D}_i$ is the product of all monic polynomials in $A$ of degree $i$ which are relatively prime to $v$.
	In other words, we remove all $v$-multiples from the usual Carlitz factorials (cf. \S\ref{subsubsec-carlitz-factorials}).
	
	Note that by definition,
	\begin{equation}          \label{eq-Di-tilde}
		\td{D}_i =
		\begin{cases}
			D_i, & i<d, \\
			D_i/(v^{q^{i-d}} D_{i-d}), & i \geq d.
		\end{cases}
	\end{equation}
	Thus, if we define $\pv: \ZZ_p \to \ZZ_p$ as
	$$
	\pv\left(\sum_{i=0}^{\infty} y_i q^i\right) := \sum_{i=0}^{\infty} y_{i+d} q^i
	$$
	where $0 \leq y_i < q$ for all $i$, then
	\begin{equation}    \label{eq-morita-carlitz-factorial}
		\td{\Pi}(n)
		= \left( \prod_{i=0}^{d-1} D_i^{n_i} \right)
		\left( \prod_{i=d}^{s} \left( \frac{D_i}{v^{q^{i-d}} D_{i-d}} \right)^{n_i} \right)
		= \frac{\Pi(n)}{\Pi(\pv(n))} \cdot v^{-\pv(n)}.
	\end{equation}
	This motivates us to define the normalized gamma values as follows.
	
	\subsubsection{Normalized \texorpdfstring{$\infty$}{infty}-adic gamma values}
	
	Given any $\alpha \in \ovl{k}^\times$, we call a system $\{\alpha^r\}_{r \in \QQ} \sbe \ovl{k}$ of roots of $\alpha$ \textit{compatible} if they satisfy $\alpha^r \alpha^s = \alpha^{r+s}$ for all $r,s \in \QQ$.
	For the $\infty$-adic case, we take any compatible system $\{\T^r\}_{r \in \QQ}$ of $\T$.
	
	Following \cite[\nopp 3.6]{thakur1991gamma}, we define the map $\pinf: \ZZ_p \to \ZZ_p$ as
	$$
	\pinf\left(\sum_{i=0}^{\infty} y_i q^i\right) := \sum_{i=0}^{\infty} iy_iq^i
	$$
	where $0 \leq y_i < q$ for all $i$.
	Then under this definition, one has $\pinf(n) = \deg \Pi(n)$ for all $n \in \NN$ (see \S\ref{subsubsec-carlitz-factorials}).
	Moreover, we will show in Lemma \ref{lem-pinf-pv-invariant} that the subgroup $\QQ \cap \ZZ_p \sbe \ZZ_p$ is invariant under both $\pinf, \pv$.
	This justifies the following definition with respect to a chosen compatible system $\{\T^r\}_{r \in \QQ}$.
	
	\begin{defn}     \label{defn-biaf}
		For $y \in \QQ \cap \ZZ_p$, we define the \textit{normalized $\infty$-adic arithmetic gamma values}
		$$
		\biaf(y)
		:= \frac{\iaf(y) \T^{\pinf(y)}}{\iaf(\pv(y)) \T^{\pinf\pv(y)}} \in \ovl{k}_\infty^\times.
		$$
	\end{defn}
	
	In particular, by \eqref{eq-morita-carlitz-factorial}, we have
	\begin{equation}     \label{eq-rem-biaf}
		\biaf(n)
		= \frac{\Pi(n)}{\Pi(\pv(n))} 
		= \td{\Pi} (n) v^{\pv(n)}
		\in k
	\end{equation}
	for all $n \in \NN$.
	
	\begin{rem}
		In order to recover the degree part of $\iaf(y)$ for all $y \in \ZZ_p$, Thakur \cite[\nopp 3.7]{thakur1991gamma} originally defined the function $\iaf(y) \T^{\pinf(y)}$ on $\ZZ_p$ with values in the completion $\widehat{k}_\infty^\times := \varprojlim k_\infty^\times / (k_\infty^\times)^{p^n}$.
		Here, since we assumed $y \in \QQ \cap \ZZ_p$, one regards $\iaf(y) \T^{\pinf(y)}$ (and also $\biaf(y)$) as values in $\ovl{k}_\infty^\times$.
	\end{rem}
	
	\subsubsection{Normalized \texorpdfstring{$v$}{v}-adic gamma values}
	
	Next, we consider the $v$-adic case.
	Recall that in Thakur's analog of the Gross-Koblitz formula \eqref{eq-gk-formula}, we defined $\lambda \in k_v(\psi(1))$ to be the unique $(q^d-1)$-st root of $-v$ such that $\lambda \equiv \psi(1) \pmod{\psi(1)^2}$.
	We then take any compatible system $\{\lambda^r\}_{r \in \QQ}$ of $\lambda$.
	
	\begin{defn}        \label{defn-bvaf}
		For $y \in \QQ \cap \ZZ_p$, we define the \textit{normalized $v$-adic arithmetic gamma values}
		$$
		\bvaf(y) := \vaf(y) \lambda^{(q^d-1)\pv(y)} (-1)^{y-q^d\pv(y)} \in \ovl{k}_v^\times.
		$$
	\end{defn}
	
	Note that if $y = \sum_{i=0}^\infty y_iq^i$ with $0 \leq y_i < q$ for all $i$, we have $y-q^d\pv(y) = \sum_{i=0}^{d-1} y_i q^i$ is an integer.
	So the power of $-1$ is defined without ambiguity.
	Moreover, by a similar discussion as in \S\ref{subsubsec-prime-to-v-carlitz-factorials}, one sees that
	\begin{equation}        \label{eq-rem-bvaf}
		\bvaf(n)
		= \frac{\Pi(n)}{\Pi(\pv(n))}
		= \td{\Pi} (n) v^{\pv(n)}
		\in k
	\end{equation}
	for all $n \in \NN$.
	
	Combining \eqref{eq-rem-biaf} and \eqref{eq-rem-bvaf}, we see that
	$$
	\biaf(n) = \bvaf(n)
	\wideforall
	n \in \NN.
	$$
	In other words, after suitable normalizations, the $\infty$-adic and $v$-adic gamma values coincide on $\NN$.
	
	\subsection{The main result}    \label{subsec-the-main-result}
	
	\subsubsection{The Gauss sums map}     \label{subsubsec-the-gauss-sums-map}
	
	For each $j \in \ZZ_{\geq0}$, we fix a compatible system of the arithmetic Gauss sum $g_j$ so that if $j \equiv j' \pmod{d}$, two systems $\{g_j^r\}_{r \in \QQ}$ and $\{g_{j'}^r\}_{r \in \QQ}$ are chosen identically.
	For a real number $y \in \RR$, we let $\ang{y}$ be its fractional part, i.e., $\ang{y}$ is the unique number such that $y \equiv \ang{y} \pmod{\ZZ}$ and $0 \leq \ang{y} < 1$.
	
	\begin{defn}     \label{defn-ari-gauss-sum-map}
		We define the function
		$$
		\ags: \QQ \cap \ZZ_p = \bigcup_{\l\in\NN} \frac{1}{q^{d\l}-1}\ZZ \lra \ovl{k}_v^\times
		$$
		as follows: if $y \in (q^{d\l}-1)^{-1}\ZZ$ for some $\l \in \NN$, we write $\ang{-y} = \sum_{i=0}^{d\l-1} y_iq^i/(q^{d\l}-1)$ with $0 \leq y_i < q$ for all $i$, and define
		$$
		\ags(y) := \prod_{i=0}^{d\l-1} g_i^{y_i/\l}.
		$$
	\end{defn}
	
	As the roots of the arithmetic Gauss sums are chosen compatibly, the function $\ags$ is well-defined on $\QQ \cap \ZZ_p$.
	In other words, it does not depend on the choice of $\l$ for each fixed $y \in \QQ \cap \ZZ_p$.
	
	\subsubsection{The main result}
	
	The following theorem is the main result of this paper, which provides an explicit $\ovl{k}$-algebra isomorphism sending the normalized $\infty$-adic gamma values to their $v$-adic counterparts through the above Gauss sums map.
	
	\begin{thm}      \label{thm-main}
		For each $\l \in \NN$, we have a $\ovl{k}$-algebra isomorphism
		$$
		\bfG: \ovl{k} \left[\biaf(y) \Bigm| y \in \frac{1}{q^\l-1} \ZZ \right]
		\isoto
		\ovl{k} \left[\bvaf(y) \Bigm| y \in \frac{1}{q^\l-1} \ZZ \right]
		$$
		given by $\bfG(\biaf(y)) := \bvaf(y) \ags(y)$.
	\end{thm}
	
	\subsubsection{A comparison with Thakur's formula}      \label{subsubsec-a-comparison-with-thakur's-formula}
	
	We now compare Theorem \ref{thm-main} with Thakur's relation for special gamma monomials.
	Given $N \in \NN$ which is relatively prime to $p$, we let $\udl{a} = \sum m_i [a_i]$ be an element in the free abelian group with basis $a_i \in N^{-1}\ZZ/\ZZ \setminus \{0\}$, and put $n(\udl{a}) := \sum m_i\ang{a_i}$.
	For $u \in (\ZZ/N\ZZ)^\times$, we set $\udl{a}^{(u)} := \sum m_i [ua_i]$.
	For any function $f$ on $N^{-1}\ZZ$ taking values in a multiplicative group, we put
	$$
	f(-\udl{a}) := \prod f(-\ang{a_i})^{m_i}.
	$$
	
	For any $y \in \QQ \cap \ZZ_p$, define
	$$
	\Pi_\infty^{\textnormal{ari},*}(y) := \iaf(y) \T^{\pinf(y)}.
	$$
	Consider the following condition:
	\begin{equation}     \label{eq-deligne-thakur-condition}
		\udl{a} = \sum m_i [a_i]
		\text{ satisfies }
		n(\udl{a}^{(q^j)})
		\text{ is an integer independent of }
		j.
	\end{equation}
	Then using his analog of Deligne's theorem \cite[Theorem 4.6]{thakur1991gamma} via interpreting arithmetic Gauss sums as Hecke characters and a general functional equation \cite[Theorem 2.8]{thakur1991gamma}, Thakur showed that if $\udl{a}$ satisfies \eqref{eq-deligne-thakur-condition} with $a_i \in (q^d-1)^{-1}\ZZ/\ZZ \setminus \{0\}$, we have
	\begin{equation}    \label{eq-deligne-thakur-formula}
		\left( \frac{\Pi_\infty^{\textnormal{ari},*}(-\udl{a})}{\td{\pi}^{n(\udl{a})}} \right)^{\Frob_P-1}
		= (-1)^{d \cdot n(\udl{a})}
		= \left( \frac{\vaf(-\udl{a})}{(-1)^{n(\udl{a})}} \right)^{i_P^{-1}}.
	\end{equation}
	Here, $P$ is a prime of $k\Fqd$ above $v$ (which is of degree one), $\Frob_P$ is the $q^d$-th power Frobenius, $\td{\pi}$ is a fixed Carlitz period, and $i_P$ is the canonical embedding from $k\Fqd$ into $k_v$.
	
	Compared with Theorem \ref{thm-main} above, for arbitrary $\udl{a}$, we let $\Frob_P$ act on $\Pi_\infty^{\textnormal{ari},*}(-\udl{a})$ by
	$$
	\Pi_\infty^{\textnormal{ari},*} (-\udl{a})^{\Frob_P}
	:= \Pi_\infty^{\textnormal{ari},*} (-\udl{a}^{(q^d)}).
	$$
	Then by Theorem \ref{thm-main}, we have
	\begin{equation}      \label{eq-bold-deligne-thakur-formula}
		\Pi_\infty^{\textnormal{ari},*} (-\udl{a})^{\Frob_P-1}
		= \biaf(-\udl{a}^{(q^d)})
		= \left( \bvaf(-\udl{a}^{(q^d)}) \ags(-\udl{a}^{(q^d)}) \right)^{i_P^{-1}}.
	\end{equation}
	Thus, our main theorem generalizes the phenomenon demonstrated in Thakur's formula \eqref{eq-deligne-thakur-formula} from a particular class of gamma monomials to the $\ovl{k}$-algebras generated by all special gamma values.
	It is also interesting to note how Gauss sums play a crucial role in both Theorem \ref{thm-main} and \eqref{eq-deligne-thakur-formula}.
	
	\begin{rem}
		When $\udl{a}$ satisfies the same conditions as in Thakur's formula, the values in \eqref{eq-bold-deligne-thakur-formula} are actually $1$.
		It is possible to adopt alternative sign normalizations for $\biaf$ and $\bvaf$ so that it reproduces the exact sign factor $(-1)^{d \cdot n(\udl{a})}$ appearing in \eqref{eq-deligne-thakur-formula}.
		Here, we choose to avoid such normalizations to ease the notations and computations.
	\end{rem}
	
	\section{Proof of the main result}     \label{sec-proof-of-the-main-result}
	
	\subsection{Preliminary lemmas}    \label{subsec-preliminary-lemmas}
	
	We begin the proof of Theorem \ref{thm-main} by establishing the standard monomial relations of the normalized gamma values.
	First, we recall the general functional equations given by Thakur \cite[Section 2]{thakur1991gamma}.
	Note that the corresponding formulas for the gamma functions $\iaf(y),\vaf(y)$ (recall \S\ref{subsubsec-arithmetic-gamma-functions}) follow immediately from Theorem \ref{thm-general-fe} and Remark \ref{rem-functional-eq}.
	
	\begin{thm}     \label{thm-general-fe}
		Let $f$ be a function on $\ZZ_p$ defined by $f(\sum y_i q^i) := \prod f_i^{y_i}$ with $0 \leq y_i < q$ for all $i$, where the convergence of $f$ is assumed.
		\begin{enumerate}
			\item (Translation) For $y \in \ZZ_p$ with $\ord_q(y) = n$,
			$$
			\frac{f(y)}{f(y-1)} = \frac{f_n}{\prod_{i=0}^{n-1} f_i^{q-1}}.
			$$
			
			\item (Strong multiplication) For any $\l \in \NN$,
			\begin{equation}       \label{eq-strong-multi}
				f\left( \sum_{i=0}^{\l-1} \frac{y_iq^i}{1-q^\l} \right)
				= \prod_{i=0}^{\l-1} f\left(\frac{q^i}{1-q^\l}\right)^{y_i},
			\end{equation}
			where $0 \leq y_i < q$ for all $i$.
		\end{enumerate}
	\end{thm}
	
	\begin{rem}       \label{rem-functional-eq}
		We mention that Theorem \ref{thm-general-fe} implies the following monomial relations of $f$ for $y \in \QQ \cap \ZZ_p$.
		\begin{enumerate}
			\item (Reflection)
			$$
			f(y) f(-1-y) = f(-1).
			$$
			
			\item (Multiplication) If $(n,q) = 1$,
			$$
			\prod_{i=1}^{n} f\left(\frac{y-i}{n}\right) = f(y-1) f(-1)^{(n-1)/2}.
			$$
			Here, if $n$ is even, we take $f(-1)^{1/2} = \prod f_i^{(q-1)/2}$.
		\end{enumerate}
	\end{rem}
	
	The following proposition follows entirely from the strong multiplication formula \eqref{eq-strong-multi}.
	Recall \S\ref{subsubsec-the-gauss-sums-map} that $\ang{\cdot}$ denotes the fractional part of a real number.
	
	\begin{prop}     \label{prop-strong-gross-koblitz}
		Let $f$ be any function on $\ZZ_p$ satisfying the strong multiplication formula \eqref{eq-strong-multi}.
		Let $\l ,n \in \NN$.
		For any $y \in (q^{n\l}-1)^{-1}\ZZ$, we write $\ang{y} = \sum_{s=0}^{n\l-1} y_sq^s/(q^{n\l}-1)$ with $0 \leq y_s < q$ for all $s$.
		Then we have
		$$
		\prod_{i=0}^{\l-1} f\left(-\ang{q^{ni}y}\right)
		= \prod_{j=0}^{n-1} f\left(\frac{q^j}{1-q^n}\right)^{e_j},
		$$
		where $e_j := \sum_{i=0}^{\l-1} y_{ni+j}$ for each $0 \leq j \leq n-1$.
	\end{prop}
	
	\begin{proof}
		For each $0 \leq i \leq \l-1$, write $\ang{q^{ni}y} = \sum_{s=0}^{n\l-1} y_s^{(i)} q^s/(q^{n\l}-1)$ where $0 \leq y_s^{(i)} < q$ for all $s$.
		Then we have
		$$
		\prod_{i=0}^{\l-1} f\left(-\ang{q^{ni}y}\right)
		\overset{\eqref{eq-strong-multi}}{=} \prod_{i=0}^{\l-1} \prod_{s=0}^{n\l-1} f\left(\frac{q^s}{1-q^{n\l}}\right)^{y_s^{(i)}}
		= \prod_{s=0}^{n\l-1} f\left(\frac{q^s}{1-q^{n\l}}\right)^{\sum_{i=0}^{\l-1} y_s^{(i)}}.
		$$
		Note that the exponent $\sum_{i=0}^{\l-1} y_s^{(i)} = e_j$ if $s \equiv j \pmod{n}$.
		Thus, we have
		$$
		\prod_{s=0}^{n\l-1} f\left(\frac{q^s}{1-q^{n\l}}\right)^{\sum_{i=0}^{\l-1} y_s^{(i)}}
		= \prod_{j=0}^{n-1} \left(\prod_{i=0}^{\l-1} f\left( \frac{q^{ni+j}}{1-q^{n\l}} \right)\right)^{e_j}
		\overset{\eqref{eq-strong-multi}}{=} \prod_{j=0}^{n-1} f\left( \frac{q^j}{1-q^n} \right)^{e_j}.
		$$ 
	\end{proof}
	
	Let $d := \deg v$.
	Recall \S\ref{subsec-normalized-arithmetic-gamma-values} that we defined two operators $\pinf, \pv: \ZZ_p \to \ZZ_p$ as
	\begin{equation}        \label{eq-pinf-pv}
		\pinf\left(\sum_{i=0}^{\infty} y_i q^i\right) := \sum_{i=0}^{\infty} iy_iq^i
		\quad
		\text{and}
		\quad
		\pv\left(\sum_{i=0}^{\infty} y_i q^i\right) := \sum_{i=0}^{\infty} y_{i+d} q^i
	\end{equation}
	where $0 \leq y_i < q$ for all $i$.
	The next two lemmas concern their basic properties.
	
	\begin{lem}    \label{lem-pinf-pv-on-Zp}
		For $a,b \in \ZZ_p$, write $a = \sum a_iq^i$ and $b = \sum b_iq^i$ where $0\leq a_i,b_i < q$ for all $i$.
		Set $a_0 + b_0 = c_0 + \delta_0 q$ where $0 \leq c_0 < q$ and $\delta_0 = 0,1$, and for each $i \geq 1$, we define $c_i, \delta_i$ recursively so that $\delta_{i-1} + a_i + b_i = c_i + \delta_iq$ where $0 \leq c_i < q$ and $\delta_i = 0,1$.
		
		(1) We have
		$$
		\pinf(a+b) = \pinf(a) + \pinf(b) + q \sum_{i=0}^{\infty} \delta_iq^i
		\wideand
		\pv(a+b) = \pv(a) + \pv(b) + \delta_{d-1}.
		$$
		In particular, if $a+b$ has no carry over base $q$ (i.e., $\delta_i=0$ for all $i$), then
		$$
		\pinf(a+b) = \pinf(a) + \pinf(b)
		\wideand
		\pv(a+b) = \pv(a) + \pv(b).
		$$
		
		(2) If $\ord_q(a) = n$, then
		$$
		-\pinf(a) = \pinf(-a) + \frac{q^{n+1}}{1-q}
		\wideand
		-\pv(a) = \pv(-a) + \eta_a
		$$
		where $\eta_a := 0$ if $d \leq n$ and $1$ otherwise.
	\end{lem}
	
	\begin{proof}
		The second assertion follows from the first by considering $a + (-a)$.
		For (1), note that $a + b = \sum c_i q^i$ where $0 \leq c_i < q$ for all $i$.
		Then by the definition of $\pinf, \pv$ (recall \eqref{eq-pinf-pv}), we have
		$$
		\pinf(a+b)
		= \sum_{i=1}^\infty i c_i q^i
		= \sum_{i=1}^\infty iq^i (a_i + b_i + \delta_{i-1} - \delta_i q),
		$$
		where one sees that
		$$
		\sum_{i=1}^\infty iq^i (\delta_{i-1} - \delta_i q)
		= q \sum_{i=0}^{\infty} \delta_i q^i.
		$$
		Also,
		$$
		\pv(a+b)
		= \sum_{i=0}^\infty c_{i+d} q^i
		= \sum_{i=0}^\infty (a_{i+d} + b_{i+d} + \delta_{i+d-1} - \delta_{i+d} q) q^i,
		$$
		and we have
		$$
		\sum_{i=0}^\infty (\delta_{i+d-1} - \delta_{i+d} q) q^i
		= \delta_{d-1}.
		$$
		This completes the proof.
	\end{proof}
	
	\begin{lem}   \label{lem-pinf-pv-invariant}
		The subgroup $\QQ \cap \ZZ_p \sbe \ZZ_p$ is invariant under $\pinf,\pv$.
	\end{lem}
	
	\begin{proof}
		Let $y \in \QQ \cap \ZZ_p$.
		Note that by Lemma \ref{lem-pinf-pv-on-Zp}, we may assume $y = q^i/(1-q^\l)$ for some $\l \in \NN$ and $0 \leq i < \l$.
		In this case, one computes that
		$$
		\pinf\left(\frac{q^i}{1-q^\l}\right)
		= \frac{iq^i}{1-q^\l} + \frac{\l q^{\l+i}}{(1-q^\l)^2}
		\in \QQ \cap \ZZ_p
		$$
		and
		$$
		\pv\left(\frac{q^i}{1-q^\l}\right)
		= \frac{q^{(i-d \bmod{\l})}}{1-q^\l}
		\in \QQ \cap \ZZ_p
		$$
		where $0 \leq (i-d \bmod{\l}) < \l$ is the unique integer congruent to $i-d$ modulo $\l$.
	\end{proof}

	\subsection{Monomial relations of normalized gamma values}   \label{subsec-monomial-relations-of-normalized-gamma-values}
		
	We now derive the standard monomial relations satisfied by the normalized gamma values (recall Definitions \ref{defn-biaf} and \ref{defn-bvaf}).
	Notably, the discrepancy between their relations is reconciled precisely by the Gauss sums map in Definition \ref{defn-ari-gauss-sum-map}.
	We will also show in \S\ref{subsec-transcendence-of-normalized-gamma-values} that in both the $\infty$-adic and $v$-adic cases, the algebraic relations among normalized gamma values are generated by these monomial relations.
	
	Recall \S\ref{subsubsec-carlitz-factorials} that for $i \geq 0$, $D_i$ is the product of all monic polynomials in $A$ of degree $i$.
	Then it is a standard fact that
	$$
	D_0 = 1
	\wideand
	D_i = \prod_{j=0}^{i-1} \left(\T^{q^i} - \T^{q^j}\right),
	\quad i \geq 1.
	$$
	We put
	$$
	L_0 := 1
	\wideand
	L_i := \prod_{j=1}^{i} \left(\T^{q^j}-\T\right),
	\quad i \geq 1.
	$$
	For convenience, we set $D_i = L_i = 1$ for $i < 0$.
	Then we have
	\begin{equation}      \label{eq-Dn-Ln}
		\frac{D_n}{\prod_{i=0}^{n-1} D_i^{q-1}}
		= L_n
	\end{equation}
	for all $n \in \ZZ$.
	All these properties can be found in \cite[Section 3]{goss1996basic}.
	
	For $y \in \ZZ_p$, we write $y = \sum_{i=0}^{\infty} y_iq^i$ with $0 \leq y_i < q$ for all $i$.
	Note that the function $\fia: \ZZ_p \to k_\infty^\times$ given by
	\begin{equation}    \label{eq-fia}
		\fia(y)
		:= \frac{\iaf(y)}{\iaf(\pv(y))}
		= \prod_{i=0}^{\infty} \left( \frac{\ovl{D}_i}{\ovl{D}_{i-d}} \right)^{y_i}
	\end{equation}
	fits into the assumption of Theorem \ref{thm-general-fe}.
	Also, for each fixed $\l \in \NN$, we put
	$$
	\fG(y) := \prod_{i=0}^{d\l-1} \left(g_i^{1/\l}\right)^{y_i}.
	$$
	Then for $y \in (q^{d\l}-1)^{-1}\ZZ$, we have
	\begin{equation}      \label{eq-ags-fG}
		\ags(y) = \fG(-\ang{-y}).
	\end{equation}
	Thus, when restricting to $(q^{d\l}-1)^{-1}\ZZ$, the function $\ags: \QQ \cap \ZZ_p \to \ovl{k}^\times$ also fits into the assumption of Theorem \ref{thm-general-fe}.
	We remark that the value $\fG(\ang{y}-1)$ does not depend on the choice of $\l$ for each fixed $y\in \QQ \cap \ZZ_p$.
	
	\begin{prop}[Translation relation]
		\label{prop-trans}
		For any $y \in \QQ \cap \ZZ_p$ with $\ord_q(y) = n$, we have
		\begin{enumerate}
			\item 
			$$
			\frac{\biaf(y)}{\biaf(y-1)}
			= \frac{L_n}{L_{n-d}}.
			$$
			
			\item
			$$
			\frac{\bvaf(y)}{\bvaf(y-1)}
			= \frac{L_n}{L_{n-d}}.
			$$
			
			\item
			$$
			\frac{\ags(y)}{\ags(y-1)} = 1.
			$$
		\end{enumerate}
	\end{prop}
	
	\begin{proof}
		(1) By Definition \ref{defn-biaf},
		$$
		\frac{\biaf(y)}{\biaf(y-1)}
		= \frac{\fia(y)}{\fia(y-1)}
		\cdot \frac{\T^{\pinf(y)-\pinf\pv(y)}}{\T^{\pinf(y-1)-\pinf\pv(y-1)}}.
		$$
		By \eqref{eq-fia}, Theorem \ref{thm-general-fe}, and \eqref{eq-Dn-Ln}, we have
		$$
		\frac{\fia(y)}{\fia(y-1)}
		= \left( \frac{\ovl{D}_n}{\ovl{D}_{n-d}} \right) \Bigg/
		\prod_{i=0}^{n-1} \left( \frac{\ovl{D}_i}{\ovl{D}_{i-d}} \right)^{q-1}
		= \frac{\ovl{L}_n}{\ovl{L}_{n-d}}.
		$$		
		On the other hand, by Lemma \ref{lem-pinf-pv-on-Zp}, one computes that
		$$
		\pinf(y)-\pinf(y-1)
		= \frac{q^{n+1}-q}{q-1}
		= \deg L_n
		$$
		and
		\begin{align*}
			\pinf\pv(y)-\pinf\pv(y-1)
			&=
			\begin{dcases}
				0, & n < d, \\
				\frac{q^{n-d+1}-q}{q-1}, & n \geq d,
			\end{dcases}   \\
			&= \deg L_{n-d}.
		\end{align*}
		This proves (1).
		
		(2) By Definition \ref{defn-bvaf},
		$$
		\frac{\bvaf(y)}{\bvaf(y-1)}
		= \frac{\vaf(y)}{\vaf(y-1)} \cdot
		\frac{\lambda^{(q^d-1)\pv(y)}}{\lambda^{(q^d-1)\pv(y-1)}} \cdot
		\frac{(-1)^{y-q^d\pv(y)}}{(-1)^{y-1-q^d\pv(y-1)}}.
		$$
		Using \eqref{eq-vaf}, Theorem \ref{thm-general-fe}, \eqref{eq-Di-tilde}, and \eqref{eq-Dn-Ln}, we have
		$$
		\frac{\vaf(y)}{\vaf(y-1)}
		= \frac{-\td{D}_n}{\prod_{i=0}^{n-1} \td{D}_i^{q-1}}
		= -\frac{L_n}{L_{n-d}} \cdot \frac{1}{v^{\kappa}}
		$$
		where $\kappa = 0$ if $n<d$ and $1$ otherwise.
		On the other hand, by Lemma \ref{lem-pinf-pv-on-Zp}, one has
		$$
		\pv(y)-\pv(y-1) = \kappa.
		$$
		This proves (2).
		
		(3) This follows immediately from Definition \ref{defn-ari-gauss-sum-map} as the map $\ags$ depends only on $y$ modulo $\ZZ$.
	\end{proof}
	
	\begin{prop}[Strong multiplication relation]   \label{prop-strong-multi}
		For any $\l \in \NN$ and $0 \leq y_i < q$ for all $i$, we have
		\begin{enumerate}
			\item
			$$
			\biaf\left( \sum_{i=0}^{\l-1} \frac{y_iq^i}{1-q^\l} \right)
			= \prod_{i=0}^{\l-1} \biaf\left(\frac{q^i}{1-q^\l}\right)^{y_i}.
			$$
			
			\item
			$$
			\bvaf \left( \sum_{i=0}^{\l-1} \frac{y_iq^i}{1-q^\l}\right)
			= \prod_{i=0}^{\l-1} \bvaf \left(\frac{q^i}{1-q^\l}\right)^{y_i}.
			$$
			
			\item
			$$
			\ags \left( \sum_{i=0}^{\l-1} \frac{y_iq^i}{1-q^\l}\right)
			= \prod_{i=0}^{\l-1} \ags \left(\frac{q^i}{1-q^\l}\right)^{y_i}.
			$$
		\end{enumerate}
	\end{prop}
	
	\begin{proof}
		Recall \eqref{eq-fia}, \eqref{eq-vaf}, and \eqref{eq-ags-fG} respectively that the functions
		$$
		\fia(y),
		\quad
		\vaf(y),
		\wideand
		\ags(y)
		\quad
		\left(y \in \frac{1}{q^{d\l}-1} \ZZ\right)
		$$
		all fit into the assumption of Theorem \ref{thm-general-fe}.
		Thus, one easily deduces their strong multiplication relations.
		On the other hand, by Lemma \ref{lem-pinf-pv-on-Zp}, we have
		$$
		\partial_{\bullet} \left( \sum_{i=0}^{\l-1} \frac{y_iq^i}{1-q^\l} \right)
		= \sum_{i=0}^{\l-1} y_i \partial_{\bullet} \left( \frac{q^i}{1-q^\l} \right) 
		$$
		for each $\bullet \in \{\infty, v\}$
		and
		$$
		\pinf \pv \left( \sum_{i=0}^{\l-1} \frac{y_iq^i}{1-q^\l} \right)
		= \pinf\left( \sum_{i=0}^{\l-1} y_i \pv \left(\frac{q^i}{1-q^\l} \right)\right)
		= \sum_{i=0}^{\l-1} y_i \pinf \pv \left( \frac{q^i}{1-q^\l} \right) 
		$$
		as the sums have no carry over base $q$.
		The results follow.
	\end{proof}
	
	\begin{prop}[Gross-Koblitz relation]   \label{prop-gk}
		For any $\l \in \NN$ and $y \in (q^{d\l}-1)^{-1}\ZZ$, we have
		\begin{enumerate}
			\item
			$$
			\prod_{i=0}^{\l-1} \biaf\left(-\ang{q^{di}y}\right) = 1.
			$$
			
			\item
			$$
			\prod_{i=0}^{\l-1} \bvaf\ags\left(-\ang{q^{di}y}\right) = 1.
			$$
		\end{enumerate}
	\end{prop}
	
	\begin{proof}
		(1) As
		$$
		\pv\left(\frac{q^j}{1-q^d}\right)
		= \frac{q^j}{1-q^d},
		$$
		one immediately sees from Definition \ref{defn-biaf} that
		$$
		\biaf\left(\frac{q^j}{1-q^d}\right) = 1
		$$
		for each $0 \leq j < d$.
		The result now follows from Propositions \ref{prop-strong-multi} and \ref{prop-strong-gross-koblitz}.
		
		(2) Similar to (1), it suffices to show that
		$$
		\bvaf \ags \left(\frac{q^j}{1-q^d}\right) = 1
		$$
		for each $0 \leq j < d$.
		By Definition \ref{defn-bvaf} and Thakur's Gross-Koblitz formula \eqref{eq-gk-formula}, we have
		$$
		\bvaf\left(\frac{q^j}{1-q^d}\right)
		= -\frac{\lambda^{q^j}}{g_j} \frac{1}{\lambda^{q^j}} (-1)^{q^j}
		= \frac{1}{g_j}.
		$$
		On the other hand, by Definition \ref{defn-ari-gauss-sum-map},
		$$
		\ags \left(\frac{q^j}{1-q^d}\right) = g_j.
		$$
		This completes the proof.
	\end{proof}
	
	\subsection{Transcendence of normalized gamma values}      \label{subsec-transcendence-of-normalized-gamma-values}
	
	Recall that our goal is to prove Theorem \ref{thm-main}.
	Note that by replacing $\l \mapsto d\l$, it is equivalent to show that for each $\l \in \NN$, we have a $\ovl{k}$-algebra isomorphism
	\begin{equation}      \label{eq-main-thm-dl}
		\bfG:
		\ovl{k} \left[\biaf(y) \Bigm| y \in \frac{1}{q^{d\l}-1} \ZZ \right]
		\isoto
		\ovl{k} \left[\bvaf(y) \Bigm| y \in \frac{1}{q^{d\l}-1} \ZZ \right]
	\end{equation}
	given by $\bfG(\biaf(y)) := \bvaf(y) \ags(y)$.
	Our next step is to use the “distribution relations” to determine the transcendence of these normalized gamma values and find concrete transcendence bases.
	
	\subsubsection{Distribution relations}       \label{subsubsec-distribution-relations}
	
	The following notations are adopted from \cite[\S 4.3.1]{wei2026algebraic} with slight adjustments.
	Let $\Aari$ be the free abelian group generated by all the elements in $(q^{d\l}-1)^{-1}\ZZ/\ZZ$.
	Then every element in $\Aari$ is written as
	$$
	\sum_{y \in (q^{d\l}-1)^{-1}\ZZ/\ZZ} n_y [y]
	\widewhere
	n_y \in \ZZ.
	$$
	We let $\Rari$ be the subgroup of $\Aari$ generated by the strong multiplication relation
	\begin{equation}        \label{eq-strong-multi-symbol}
		[y] - \sum_{i=0}^{d\l-1} y_i \left[\frac{q^i}{1-q^{d\l}}\right]
	\end{equation}
	for all $y \in (q^{d\l}-1)^{-1}\ZZ/\ZZ$ with $\ang{-y} = \sum_{i=0}^{d\l-1} y_iq^i/(q^{d\l}-1)$ where $0 \leq y_i < q$ for all $i$, and the Gross-Koblitz relation
	\begin{equation}       \label{eq-gk-symbol}
		\sum_{i=0}^{\l-1} \left[-\ang{q^{di}y}\right]
		\wideforall
		y \in \frac{1}{q^{d\l}-1}\ZZ/\ZZ.
	\end{equation}
	We put $\Uari := \Aari / \Rari$.
	
	For $0 \leq i < d\l$, let
	$$
	\bfy_i := \frac{q^i}{1-q^{d\l}}.
	$$
	We let $\Fari$ be the quotient of $\Aari$ by the subgroup generated by the strong multiplication relation \eqref{eq-strong-multi-symbol}.
	Then by \cite[\S4.3.1]{wei2026algebraic}, $\Fari$ is a free abelian group of rank $d\l$ with basis
	$$
	\left\{ [\bfy_i] \mid 0 \leq i < d\l \right\}.
	$$
	
	\begin{rem}
		In fact, Wei showed that $\FariQ := \QQ \otimes_\ZZ \Fari$ is a $\QQ$-vector space of dimension $d\l$ with basis $\{[\bfy_i] \mid 0 \leq i < d\l\}$, which is called the \textit{universal distribution associated to the arithmetic diamond bracket relations}.
		We mention that to deal with the transcendence problem of special gamma values, it is enough to compute the dimension of the distribution relations.
		Here, since we are considering the algebras generated by special gamma values, a careful analysis of the group structures is necessary.
	\end{rem}
	
	\subsubsection{The \texorpdfstring{$\infty$}{infty}-adic case}
	
	For the $\infty$-adic case, we let $\Vari \sbe \Fari$ be the subgroup generated by
	$$
	\left\{[y] - [\pv(y)] \Bigm| y \in \frac{1}{q^{d\l}-1}\ZZ/\ZZ\right\},
	$$
	which is free of rank $\leq d\l$.
	
	\begin{prop}     \label{prop-rank-basis-of-V}
		$\Vari$ is a free abelian group of rank $d\l-d$ with basis
		$$
		\left\{ [\bfy_i] - [\pv(\bfy_i)] \mid d \leq i < d\l \right\}.
		$$
	\end{prop}
	
	\begin{proof}
		Since $\{ [\bfy_i] \mid 0 \leq i < d\l \}$ is a basis of $\Fari$, one checks that the set given by the proposition is linearly independent.
		To show that $\Vari$ is generated by this set, note that by the strong multiplication relation \eqref{eq-strong-multi-symbol}, it suffices to generate $[\bfy_i] - [\pv(\bfy_i)]$ for $0 \leq i < d$.
		In this case, one has
		$$
		[\bfy_i] - [\pv(\bfy_i)]
		= -\sum_{j=1}^{\l-1} \left( [\bfy_{i+jd}] - [\pv(\bfy_{i+jd})] \right).
		$$
		This completes the proof.
	\end{proof}
	
	By \cite[Theorem 3.3.5 and Corollary 4.3.3]{wei2026algebraic}, there exists an injective $\QQ$-linear homomorphism
	$$
	\widehat{\Pi}^{\textnormal{ari}}_\infty: \FariQ \lra \ovl{k}_\infty^\times / \ovl{k}^\times,
	\quad
	\widehat{\Pi}^{\textnormal{ari}}_\infty([y]) := \iaf(-\ang{-y}),
	$$
	called the \textit{arithmetic gamma distribution}.
	Since $\biaf(y)$ differs from $\iaf(y) / \iaf(\pv(y))$ by an algebraic multiple (recall Definition \ref{defn-biaf}), by Proposition \ref{prop-rank-basis-of-V}, we obtain the following transcendence result for the normalized $\infty$-adic gamma values.
	
	\begin{cor}      \label{cor-trans-degree-of-inf}
		We have
		$$
		\ovl{k} \left( \biaf(y) \Bigm| y \in \frac{1}{q^{d\l}-1} \ZZ \right)
		= \ovl{k} \left( \biaf(\bfy_i) \bigm| d \leq i < d\l\right),
		$$
		and the transcendence degree over $\ovl{k}$ is $d\l-d$.
		In particular, the special gamma values
		$$
		\left\{\biaf(\bfy_i) \bigm| d \leq i < d\l\right\}
		$$
		are algebraically independent.
	\end{cor}
	
	\subsubsection{The \texorpdfstring{$v$}{v}-adic case}           \label{subsubsec-the-v-adic-case}
	
	In view of Definition \ref{defn-bvaf}, the transcendence of the normalized $v$-adic gamma values has been fully settled in \cite{cwy2026vadic}.
	Below, we find a concrete transcendence base.
	
	Recall \S\ref{subsubsec-distribution-relations} that $\Uari$ is the the quotient of $\Aari$ by the subgroup $\Rari$ generated by the strong multiplication relation \eqref{eq-strong-multi-symbol} and the Gross-Koblitz relation \eqref{eq-gk-symbol}.
	
	\begin{prop}     \label{prop-rank-basis-of-u}
		$\Uari$ is a free abelian group of rank $d\l-d$ with basis
		$$
		\left\{ [\bfy_i] \mid d \leq i < d\l \right\}.
		$$
	\end{prop}

	\begin{proof}
		By the strong multiplication relation \eqref{eq-strong-multi-symbol} and the Gross-Koblitz relation \eqref{eq-gk-symbol}, we have
		$$
		\mathbf{0}
		= \sum_{i=0}^{\l-1} \left[-\ang{q^{di}y}\right]
		= \sum_{j=0}^{d-1} \left(\sum_{i=0}^{\l-1} y_{di+j} \right) \left( \sum_{i=0}^{\l-1} [\bfy_{di+j}] \right)
		\in \Uari
		$$
		for all $y \in (q^{d\l}-1)^{-1}\ZZ/\ZZ$ with $\ang{y} = \sum_{i=0}^{d\l-1} y_iq^i / (q^{d\l}-1)$ where $0\leq y_i < q$ for all $i$.
		In particular, taking $y = q^{j_0} / (q^{d\l}-1)$ with $0 \leq j_0 < d$ implies that
		$$
		\sum_{i=0}^{\l-1} [\bfy_{di+j_0}] = \mathbf{0}.
		$$
		Thus, $\Uari$ is generated by $\{ [\bfy_i] \mid d \leq i < d\l \}$ and the rank is bounded above by $d\l-d$.
		
		On the other hand, let $\EariQ$ be the subspace of $\FariQ = \QQ \otimes_\ZZ \Fari$ generated by the Gross-Koblitz relation \eqref{eq-gk-symbol}.
		Note that $\EariQ$ is generated by the linearly independent set
		$$
		\left\{ \sum_{i=0}^{\l-1} [\bfy_{di+j}] \Bigm| 0 \leq j < d \right\}.
		$$
		In particular, from the isomorphism
		$$
		\UariQ := \QQ \otimes_\ZZ \Uari \simeq \FariQ / \EariQ,
		$$
		this implies that $\dim_\QQ \UariQ = d\l - d$.
		Hence, $\Uari$ is a free abelian group of rank $d\l-d$ with basis $\{ [\bfy_i] \mid d \leq i < d\l \}$.
	\end{proof}
	
	By \cite[Theorem 4.2.1]{cwy2026vadic}, there exists an injective $\QQ$-linear homomorphism
	$$
	\widehat{\Pi}^{\textnormal{ari}}_v: \UariQ \lra  \ovl{k}_v^\times / \ovl{k}^\times,
	\quad
	\widehat{\Pi}^{\textnormal{ari}}_v([y]) := \vaf(-\ang{-y}),
	$$
	called the \textit{$v$-adic arithmetic gamma distribution}.
	Then by Proposition \ref{prop-rank-basis-of-u}, we obtain a transcendence base of the normalized $v$-adic gamma values.
	
	\begin{cor}
		We have
		$$
		\ovl{k} \left( \bvaf(y) \Bigm| y \in \frac{1}{q^{d\l}-1} \ZZ \right)
		= \ovl{k} \left( \bvaf(\bfy_i) \bigm| d \leq i < d\l\right)
		$$
		and the special gamma values
		$$
		\left\{\bvaf(\bfy_i) \bigm| d \leq i < d\l\right\}
		$$
		are algebraically independent.
	\end{cor}
	
	\subsection{The algebra isomorphism}     \label{subsec-the-algebra-isomorphism}
	
	We now finish the proof of Theorem \ref{thm-main} by establishing the desired $\ovl{k}$-algebra isomorphism \eqref{eq-main-thm-dl}.
	Our strategy is to identify the algebras in question as the quotient of a polynomial ring over $\ovl{k}$ by the strong multiplication and Gross-Koblitz relations.
	
	\begin{proof}
		Consider the polynomial ring
		$$
		\ovl{k}\left[X_{[y]} \Bigm| y \in \frac{1}{q^{d\l}-1}\ZZ/\ZZ \right]
		$$
		where the independent variables $X_{[y]}$ are indexed by the basis of $\Aari$.
		Let $I_\l$ be the ideal generated by
		$$
		X_{[y]} - \prod_{i=0}^{d\l-1} X_{[\bfy_i]}^{y_i}
		$$
		for all $y \in (q^{d\l}-1)^{-1}\ZZ/\ZZ$ with $\ang{-y} = \sum_{i=0}^{d\l-1} y_iq^i/(q^{d\l}-1)$ where $0 \leq y_i < q$ for all $i$, and
		$$
		\prod_{i=0}^{\l-1} X_{[-\ang{q^{di}y}]} - 1
		\wideforall
		y \in \frac{1}{q^{d\l}-1}\ZZ/\ZZ.
		$$
		Then by the strong multiplication relation (Proposition \ref{prop-strong-multi}) and the Gross-Koblitz relation (Proposition \ref{prop-gk}) of $\biaf$, the evaluation map
		$$
		\ev_\infty: X_{[y]} \longmapsto \biaf(-\ang{-y})
		$$
		factors through
		$$
		\ev_\infty: \ovl{k}\left[X_{[y]} \Bigm| y \in \frac{1}{q^{d\l}-1}\ZZ/\ZZ \right] \bigg/ I_\l
		\ltra
		\ovl{k} \left[\biaf(y) \Bigm| y \in \frac{1}{q^{d\l}-1}\ZZ \cap (-1,0] \right].
		$$
		Also, by Corollary \ref{cor-trans-degree-of-inf}, it induces the $\ovl{k}$-algebra isomorphism
		$$
		\ovl{k} \left[ X_{[\bfy_i]} \bigm| d \leq i < d\l \right]
		\simeq \ovl{k} \left[\biaf(\bfy_i) \bigm| d \leq i < d\l\right].
		$$
		Hence, by Propositions \ref{prop-rank-basis-of-u} and \ref{prop-rank-basis-of-V}, we have the following commutative diagram:
		$$
		\begin{tikzcd}
			\ovl{k} \left[X_{[y]}^{\pm1} \Bigm| y \in \displaystyle \frac{1}{q^{d\l}-1}\ZZ/\ZZ \right] / I_\l   \arrow[r, "\ev_\infty", two heads]
			& \ovl{k} \left[\biaf(y)^{\pm1} \Bigm| y \in \displaystyle \frac{1}{q^{d\l}-1}\ZZ \cap (-1,0] \right]   \\
			\ovl{k} \left[ X_{[\bfy_i]}^{\pm1} \bigm| d \leq i < d\l \right] \arrow[r, "\sim"] \arrow[u, two heads]
			& \ovl{k} \left[\biaf(\bfy_i)^{\pm1} \bigm| d \leq i < d\l\right] \arrow[u, equal].
		\end{tikzcd}
		$$
		In particular, this implies that the upper map $\ev_\infty$ is in fact injective.
		Hence, we obtain the $\ovl{k}$-algebra isomorphism
		\begin{equation}     \label{eq-ev_infty}
			\ev_\infty: \ovl{k} \left[X_{[y]} \Bigm| y \in \frac{1}{q^{d\l}-1}\ZZ/\ZZ \right] \bigg/ I_\l  \isoto \ovl{k} \left[\biaf(y) \Bigm| y \in \frac{1}{q^{d\l}-1}\ZZ \cap (-1,0] \right].
		\end{equation}
		
		On the other hand, by the monomial relations of $\bvaf \ags$ in Propositions \ref{prop-strong-multi}, \ref{prop-gk}, and the results in \S\ref{subsubsec-the-v-adic-case}, the above argument also works for the evaluation map
		$$
		\ev_v: X_{[y]} \longmapsto \bvaf(-\ang{-y}) \ags(-\ang{-y}).
		$$
		In particular, we have the $\ovl{k}$-algebra isomorphism
		\begin{equation}      \label{eq-ev_v}
			\ev_v: \ovl{k} \left[X_{[y]} \Bigm| y \in \frac{1}{q^{d\l}-1}\ZZ/\ZZ \right] \bigg/ I_\l  \isoto \ovl{k} \left[\bvaf(y) \Bigm| y \in \frac{1}{q^{d\l}-1}\ZZ \cap (-1,0] \right].
		\end{equation}
		
		Combining \eqref{eq-ev_infty} and \eqref{eq-ev_v}, we have
		$$
		\ovl{k} \left[\biaf(y) \Bigm| y \in \frac{1}{q^{d\l}-1}\ZZ \cap (-1,0] \right]
		\isoto
		\ovl{k} \left[\bvaf(y) \Bigm| y \in \frac{1}{q^{d\l}-1}\ZZ \cap (-1,0] \right]
		$$
		sending $\biaf(y)$ to $\bvaf(y) \ags(y)$.
		Finally, by the translation relations (Proposition \ref{prop-trans}), this extends to the desired $\ovl{k}$-algebra isomorphism
		$$
		\ovl{k} \left[\biaf(y) \Bigm| y \in \frac{1}{q^{d\l}-1}\ZZ \right]
		\isoto
		\ovl{k} \left[\bvaf(y) \Bigm| y \in \frac{1}{q^{d\l}-1}\ZZ \right]
		$$
		sending $\biaf(y)$ to $\bvaf(y) \ags(y)$.
		This completes the proof.
	\end{proof}
	
	\subsection*{Acknowledgments}
	
	The author is deeply grateful to Fu-Tsun Wei for his invaluable guidance and constant encouragement during the preparation of this work.
	He also thanks the National Science and Technology Council for generous financial support over the past few years.
	
	\printbibliography

\end{document}